\documentclass[12pt,a4paper,final, dvipsnames,reqno]{amsart}

\input xy
\xyoption{all}

\DeclareMathAlphabet{\mathpzc}{OT1}{pzc}{m}{it}

\usepackage{amsfonts,amsmath,amssymb,indentfirst,mathrsfs,amscd}
\usepackage[mathscr]{eucal}
\usepackage{graphicx}
\usepackage{amsthm}

\usepackage{color}
\usepackage[utf8]{inputenc}
\usepackage[english]{babel}
\usepackage{booktabs}
\usepackage{multirow}
\usepackage{verbatim}
\usepackage{hyperref}
\usepackage{subcaption}

\DeclareMathOperator{\Hom}{Hom\,}           %Hom%

\usepackage[all]{xy}
\usepackage{graphicx}
\usepackage{stmaryrd}
\usepackage{enumitem}

\usepackage{empheq}
\usepackage{color}
\usepackage{xcolor}

\author[\'A. Gonz\'alez-Prieto]{\'Angel Gonz\'alez-Prieto}
\address{Departamento de Matem\'aticas, Facultad de Ciencias, Universidad Aut\'onoma de Madrid. C.\ Francisco Tomás y Valiente, 7, 28049 Madrid, Spain.}
\address{Instituto de Ciencias Matem\'aticas (CSIC-UAM-UC3M-UCM), C.\ Nicol\'as Cabrera 15, 28049 Madrid, Spain}
\email{angel.gonzalezprieto@uam.es}

\author[M. Logares]{Marina Logares}
\address{Facultad de Ciencias Matem\'aticas, Universidad Complutense  de  Madrid, Plaza Ciencias 3, 28040 Madrid Spain.}\email{mlogares@ucm.es}

\author[V. Mu\~{n}oz]{Vicente Mu\~{n}oz}
\address{Departamento de \'Algebra, Geometr\'ia y Topolog\'ia, Facultad de Ciencias, Universidad de M\'alaga, Campus de Teatinos s/n, 29071 Málaga, Spain}
\email{vicente.munoz@uma.es}

\title[Representation variety of torus knots for affine groups]{Motive of the representation varieties of torus knots for low rank affine groups}


\usepackage{amsmath}
\begin{document}

\newtheorem{theorem}{Theorem}[section]
\newtheorem{proposition}[theorem]{Proposition}
\newtheorem{lemma}[theorem]{Lemma}
\newtheorem{corollary}[theorem]{Corollary}
\newtheorem{conjecture}{Conjecture}
\newtheorem*{theorem*}{Theorem}

\theoremstyle{definition}
\newtheorem{definition}[theorem]{Definition}
\newtheorem{example}[theorem]{Example}
\newtheorem{as}{Assumption}
\theoremstyle{remark}
\newtheorem{remark}[theorem]{Remark}
\theoremstyle{remark}
\newtheorem*{prf}{Proof}

\newcommand{\cA}{\mathcal{A}}
\newcommand{\cC}{\mathcal{C}}
\newcommand{\cD}{\mathcal{D}}
\newcommand{\cE}{\mathcal{E}}
\newcommand{\cF}{\mathcal{F}}
\newcommand{\cG}{\mathcal{G}} %Gauge group%
\newcommand{\cI}{\mathcal{I}} %Gauge group%
\newcommand{\cO}{\mathcal{O}} %Holomorphic functions sheaf%
\newcommand{\cM}{\mathcal{M}} %Moduli space% %moduli of parabolic bundles% %moduli of U(p,q) bundles%
\newcommand{\cN}{\mathcal{N}} %Space of minimal points of the Morse function%
\newcommand{\cP}{\mathcal{P}} %Moduli of K(D) pairs%
\newcommand{\cQ}{\mathcal{Q}} %Moduli of K(D) pairs%
\newcommand{\cS}{\mathcal{S}} %Moduli of solutions of Hitchin's equations, contructed by Konno%
\newcommand{\cU}{\mathcal{U}} %Moduli of stable U(p,q) parabolic Higgs bundles%
\newcommand{\cJ}{\mathcal{J}}
\newcommand{\cX}{\mathcal{X}}
\newcommand{\cT}{\mathcal{T}}
\newcommand{\cV}{\mathcal{V}}
\newcommand{\cW}{\mathcal{W}}
\newcommand{\cB}{\mathcal{B}}
\newcommand{\cR}{\mathcal{R}}
\newcommand{\cH}{\mathcal{H}}
\newcommand{\cZ}{\mathcal{Z}}

\newcommand{\QQ}{\mathbb{Q}} %Rational numbers%
\newcommand{\FF}{\mathbb{F}} %Rational numbers%
\newcommand{\PP}{\mathbb{P}} %projective space%
\newcommand{\HH}{\mathbb{H}} %Hypercohomology, quaternions..%
\newcommand{\RR}{\mathbb{R}} %Real numbers%
\newcommand{\ZZ}{\mathbb{Z}} %Integer numbers%
\newcommand{\NN}{\mathbb{N}} %Natural numbers%
\newcommand{\DD}{\mathbb{D}} %Natural numbers%
\newcommand{\CC}{\mathbb{C}} %Natural numbers%

\newcommand{\too}{\longrightarrow}
\newcommand{\G}{\Gamma}         %group%
\newcommand{\id}{\mathrm{id}}
\newcommand{\Ker}{\textrm{Ker}\,}
\newcommand{\coker}{\textrm{coker}\,}
\newcommand{\x}{\times}
\newcommand\HS[1]{\mathbf{HS}^{#1}}
\newcommand\MHS[1]{\mathbf{MHS}}
\newcommand\Var[1]{\mathbf{Var}_{#1}}
\newcommand\K[1]{\mathrm{K}#1}
\newcommand\AGL[1]{\mathrm{AGL}_{#1}}
\newcommand\GL[1]{\mathrm{GL}_{#1}}
\newcommand\SL[1]{\mathrm{SL}_{#1}}
\newcommand\PGL[1]{\mathrm{PGL}_{#1}}
\newcommand\Rep[1]{\mathfrak{X}_{#1}}
\newcommand\Id{\mathrm{Id}}
\newcommand\Char[1]{\cR_{#1}}
\newcommand{\tr}{\textrm{tr}\,}             %Trace Tr%
\newcommand\Gr{\textrm{Gr}}
\newcommand{\red}{\mathrm{red}}
\newcommand{\irr}{\mathrm{irr}}

\begin{abstract}
We compute the motive of the variety of representations of the torus knot of type $(m,n)$ into the affine groups
$\AGL{1}(\CC)$ and $\AGL{2}(\CC)$. For this, we stratify the varieties and show that the motives lie
in the subring generated by the Lefschetz motive $q=[\CC]$.
\end{abstract}
\null
\vspace{-1.1cm}
\maketitle

\vspace{-1.5cm}

%%%%%%%%%%%%%%%%%%%%%%%%%%%%%%%%%%%%%%%%%%%%%%%%%%%%%%%%%%%%%%%%
%%%%%%%%%%%%%%%%%%% SECTION: INTRODUCTION %%%%%%%%%%%%%%%%%%%%%%
%%%%%%%%%%%%%%%%%%%%%%%%%%%%%%%%%%%%%%%%%%%%%%%%%%%%%%%%%%%%%%%%

\section{Introduction}
\let\thefootnote\relax\footnotetext{\noindent \emph{2020 Mathematics Subject Classification}. Primary:
 14C30. % Hodge theory
 Secondary:
 57R56, % TQFT
 14L24, % Geometric Invariant Theory
 14D21. % Applications of moduli in mathematical physics

\emph{Key words and phrases}: torus knots, character varieties, affine group, Grothendieck ring.}

Since the foundational work of Culler and Shalen \cite{CS}{}, 
the varieties of $\SL{2}(\CC)$-characters have been extensively studied. 
Given a manifold $M$, the variety of representations of
$\pi_1(M)$ into $\SL{2}(\CC)$ and the variety of characters of such representations both
contain information on the topology of $M$. It is especially
interesting for $3$-dimensional manifolds, where the fundamental
group and the geometrical properties of the manifold are 
strongly related. This can be used to study knots $K\subset S^3$, by analysing the
$\SL{2}(\CC)$-character variety of the fundamental group of the knot complement
$S^3-K$ (these are called \emph{knot groups}). 

For a very different reason, the case of fundamental groups of 
surfaces has also been extensively analysed \cite{HT, Hi, LMN, MM, MM2}{}, 
in this situation focusing more on geometrical properties of the moduli
space in itself (cf.\ non-abelian Hodge theory).

Much less is known of the character varieties for other groups. The character varieties for $\SL{3}(\CC)$ for
free groups have been described in \cite{LM}{}. In the case of $3$-manifolds,
little has been done. In this paper, we focus in the case of the torus knots $K_{m,n}$ 
for coprime $m,n$, which are the first family of knots where the computations are
rather feasible. The case of $\SL{2}(\CC)$-character varieties of torus knots was carried
out in \cite{MO, Mu}{}. 
For $\SL{3}(\CC)$, it has been carried out by Mu\~noz and Porti in \cite{MP}{}. 
The case of $\SL{4}(\CC)$ has been computed by two of the authors of the current paper through a computer-assisted proof in \cite{GPM}{}. 

The group $\SL{r}(\CC)$ is reductive, which allows to use Geometric Invariant Theory (GIT) to 
define the moduli of representations, the so-called \emph{character variety}. In \cite{GPLM} we started the analysis of character varieties
for the first non-reductive groups, notably computing by three different methods (geometric, arithmetic and through a Topological Quantum Field Theory)
the motive of the variety of representations for a surface group into the rank one affine group $\AGL{1}(\CC)$. 

In the current work, we study the variety of representations of the torus knot $K_{m,n}$ into the
affine groups $\AGL{1}(\CC)$ and $\AGL{2}(\CC)$. We prove the following result:

\begin{theorem}\label{thm:main}
Let $m,n \in \NN$ with $\gcd(m,n)=1$. The motives of the $\AGL{1}(\CC)$ and $\AGL{2}(\CC)$-representation 
variety of the $(m,n)$-torus knot in the Grothendieck ring of complex algebraic varieties are:
\begin{align*}
	\big[\Rep{m,n}(&\AGL{1}(\CC)) \big] = (mn-m-n+2)(q^2-q). \\
	\big[\Rep{m,n}(&\AGL{2}(\CC))\big] = q^6 - 2q^4 + q^3 + \left[\Rep{m,n}^{\irr}(\GL{2}(\CC))\right]q^2  \\
	&  + (m-1)(n-1)\left((q - 1)(q - 2)\frac{(m-2)(n-2)q+mn-4}{4} + (q + 1) - 2\right)(q^5-q^3) \\
	&  + \frac{(m-1)(n-1)(mn-m-n)}{2}\left(q - 1\right)(q^5-q^3).  
	\end{align*}
Here, $q = [\CC] \in \K{\Var{\CC}}$ denotes the Lefschetz motive, and $\left[\Rep{m,n}^{\irr}(\GL{2}(\CC))\right] =$
 $$
  =\left\{\begin{array}{ll} (q^3-q) \frac14 (m-1)(n-1)(q-2) (q-1),  & m,n \text{ both odd}, \\
   (q^3-q) \left(\frac14 (n-2)(m-1)(q-2) +\frac12 (m-1)(q-1)\right) (q-1), & m \text{ odd, } n \text{ even,} \\ 
  (q^3-q) \left(\frac14 (n-1)(m-2)(q-2) +\frac12 (n-1)(q-1)\right) (q-1), \quad & m \text{ even, } n \text{ odd.} 
  \end{array}\right.
  $$
\end{theorem}

\noindent \textbf{Acknowledgements.} The first author was partially supported by MICINN (Spain) grant PID2019-106493RB-I00.
The third author was partially supported by MINECO (Spain) grant PGC2018-095448-B-I00.

\section{Basic notions} %The Grothendieck ring of algebraic varieties}

%%%%%%%%%%%%%%%%%%%%%%%%%%%%%%%%%%%%%%%%%%%
\subsection{Representation varieties of torus knots}\label{sec:rep-var}
%%%%%%%%%%%%%%%%%%%%%%%%%%%%%%%%%%%%%%%%%%%

Let $\G$ be a finitely presented group, and let $G$ be a complex algebraic group. %=\SL(r,\CC)$, $\GL(r,\CC)$ or $\PGL(r,\CC)$. 
A \textit{representation} of $\G$ in $G$ is a homomorphism $\rho: \G\to G$.
Consider a presentation $\G=\langle x_1,\ldots, x_k \,|\, r_1,\ldots, r_s \rangle$. Then $\rho$ is completely
determined by the $k$-tuple $(A_1,\ldots, A_k)=(\rho(x_1),\ldots, \rho(x_k))$
subject to the relations $r_j(A_1,\ldots, A_k)=\Id$, $1\leq j \leq s$. The \textit{representation variety} is
 \begin{eqnarray*}
 \Rep{\G}(G) &=& \Hom(\G, G) \\
  &=& \{(A_1,\ldots, A_k) \in G^k \, | \,
 r_j(A_1,\ldots, A_k)=\Id, \,  1\leq j \leq s \}\subset G^{k}\, .
 \end{eqnarray*}
Therefore $\Rep{\G}(G)$ is an affine algebraic set.

Suppose in addition that $G$ is a linear group, say $G \subset \GL{r}(\CC)$. A representation $\rho$ is
\textit{reducible} if there exists some proper 
subspace $V\subset \CC^r$ such that for  all $g\in G$ we have 
$\rho(g)(V)\subset V$;  otherwise $\rho$ is
\textit{irreducible}. This distinction induces a natural stratification of the representation variety into
its irreducible and reducible parts $\Rep{\G}(G) = \Rep{\G}^{\irr}(G) \sqcup \Rep{\G}^{\red}(G)$.
%We say that two representations $\rho$ and $\rho'$ are
%equivalent if there exists $P\in G$ such that $\rho'(g)=P^{-1} \rho(g) P$,
%for every $g\in \G$. The moduli space of representations (also called \textit{character variety})
%is the GIT quotient under the conjugacy action
% $$
% \Char{\G}(G) = \Rep{\G}(G) \sslash G \, .
% $$
%If $G$ is reductive, then by definition of GIT quotient for an affine variety, if we write
%$\Rep{\G}(G)=\Spec A$, then $\Char{\G}(G)=\Spec A^{G}$ where $A^{G} \subset A$ is the subring of $G$-fixed elements.

Let $T^2=S^1 \times S^1$ be the $2$-torus and consider the standard embedding
$T^2\subset S^3$. Let $m,n$ be a pair of coprime positive integers. Identifying
$T^2$ with the quotient $\RR^2/\ZZ^2$, the image of the straight line $y=\frac{m}{n}
x$ in $T^2$ defines the \textit{torus knot} of type $(m,n)$, which we shall denote
as $K_{m,n}\subset S^3$ (see Chapter 3 in \cite{Ro}{}).
For a knot $K\subset S^3$, we denote by $\G_K$ the fundamental group of the exterior
$S^3-K$ of the knot. It is known that
 $$
  \G_{m,n}= \G_{K_{m,n}} \cong \langle x,y \, | \, x^n= y^m \,\rangle \,.
 $$
Therefore the variety of representations of the torus knot of type $(m,n)$ is described as
 $$
  \Rep{m,n}(G)=\Rep{\G_{m,n}}(G)=\{ (A,B) \in G^2 \,|\, A^n=B^m\}. 
  $$

In this work, we shall focus on the case $G = \AGL{r}(\CC)$, the group of affine automorphisms of the complex $r$-dimensional affine space.

\subsection{The Grothendieck ring of algebraic varieties}

%From a (skeletally small) abelian category $\cA$, 
%it is possible to construct an abelian group, known as the Grothendieck group of $\cA$. It is the abelian group $\K{\cA}$
%generated by the isomorphism classes $[A]$ of objects $A\in \cA$, subject to the relations that whenever there exists a short exact sequence 
%$
%0\to B\to A\to C\to 0
%$
%we declare $[A]=[B]+[C]$.
%Furthermore, if our abelian category is provided with a tensor product, i.e. $\cA$ is %symmetric 
%monoidal, and the functors $-\otimes A:\cA\rightarrow \cA$ and $A\otimes -: \cA\rightarrow \cA$ are exact, then $\K{\cA}$ inherits a ring
%structure by $[A]\cdot[B]= [A\otimes B]$, under which it is called the \emph{Grothendieck ring} of $\cA$. The elements
%$[A] \in \K{\cA}$ are usually referred to as virtual classes.

Take the category of complex algebraic varieties with regular morphisms $\Var{\CC}$.
We can construct its Grothendieck group, $\K{\Var{\CC}}$, as the abelian group generated by
isomorphism classes of algebraic varieties with the relation that $[X]=[Y]+[U]$ if $X=Y\sqcup U$, with $Y\subset X$ a closed subvariety. 
The cartesian product of varieties also provides $\K{\Var{\CC}}$ with a ring structure, as $[X]\cdot [Y]=[X\x Y]$. 
The elements of $\K{\Var{\CC}}$ are usually referred to as virtual classes.
A very important element of $\K{\Var{\CC}}$
is the class of the affine line, $q = [\CC]$, the so-called \emph{Lefschetz motive}. 

%\begin{remark} \label{rem:zerodivisor}
%Despite the simplicity of its definition, the ring structure of $\K{\Var{\CC}}$ is widely unknown. In particular, for almost fifty years it was an open problem whether it is an integral domain. Indeed, the answer if no and, more strikingly, the Lefschetz motive $q$ is a zero divisor \cite{Borisov:2014}.
%\end{remark}

Virtual classes are well-behaved with respect to two typical geometric situations that we will encounter in the upcoming sections. A proof of the
following facts can be found for instance in Section 4.1 of \cite{GP:2018b}{}.
\begin{itemize}
	\item Let $E \to B$ be a regular morphism that is a locally trivial bundle in the Zariski topology with fiber $F$. In this situation, we have
	that in $\K\Var{\CC}$
	$$
		[E]=[F]\cdot [B].
	$$
	\item Suppose that $X$ is an algebraic variety with an action of $\ZZ_2$.
Setting $[X]^+ = [X / \ZZ_2]$ and $[X]^- = [X] - [X]^+$, we have the formula
  \begin{equation}\label{eqn:2bis}
 [X \times Y]^{+} = [X]^+[Y]^+ + [X]^-[Y]^- 
  \end{equation}
for two varieties $X,Y$ with $\ZZ_2$-actions. 
\end{itemize}

\begin{example}\label{ex:calculations-kvar}
Consider the fibration $\CC^2 - \CC \to \GL{2}(\CC) \to \CC^{2}-\{(0,0)\}$, $f \mapsto f(1,0)$. It is locally trivial in the Zariski topology,
and therefore $[\GL{2}(\CC)]=[\CC^2-\CC]\cdot[\CC^{2}-\{(0,0)\}] = (q^2-q)(q^2-1) = q^4-q^3-q^2+q$. Analogously, the quotient map defines a
locally trivial fibration $\CC^* = \CC - \{0\} \to \GL{2}(\CC) \to \PGL{2}(\CC)$, so $[\PGL{2}(\CC)] = q^3-q$.
\end{example}

We have the following computation that we will need later.

\begin{lemma}\label{lem;ac}
Let $\ZZ_2$ act on $\CC^2$ by exchange of coordinates. Then 
$[(\CC^*)^2 - \Delta]^+ = (q-1)^2$, $[(\CC^*)^2 - \Delta]^- = -q+1$, where $\Delta$ denotes the diagonal.

Also let $X=\GL{2}(\CC) / \GL{1}(\CC) \times \GL{1}(\CC)$, and $\ZZ_2$ acting by exchange of columns in $\GL{2}(\CC)$.
Then $[X]^+=q^2$ and $[X]^-=q$.
\end{lemma}

\begin{proof}
 The quotient $\CC^2/\ZZ_2$ is parametrized by $s=x+y$, $p=xy$, where $(x,y)$ are the coordinates of $\CC^2$. Then
 $((\CC^*)^2-\Delta)/\ZZ_2$ is given by the equations $p\neq 0$, $4p\neq s^2$. Therefore 
 $[(\CC^*)^2 - \Delta]^+ =[((\CC^*)^2-\Delta)/\ZZ_2]=q^2-q - (q-1)=(q-1)^2$, and  
 $[(\CC^*)^2 - \Delta]^-= [(\CC^*)^2 - \Delta] - [(\CC^*)^2 - \Delta]^+ =(q-1)^2-(q-1)- (q-1)^2=-q+1$.
 
 For the second case, note that $X=\PP^1\x\PP^1-\Delta$, and $\ZZ_2$ acts by exchange of coordinates. The whole quotient is
 $(\PP^1\x \PP^1)/\ZZ_2=\mathrm{Sym}^2(\PP^1)=\PP^2$. The diagonal goes down to a smooth conic (the completion of
 $4p=s^2$), hence $[X]^+=[X/\ZZ_2]=[(\PP^1\x \PP^1-\Delta)/\ZZ_2]=q^2+q+1-(q+1)=q^2$. Also $[X]=(q+1)^2-(q+1)=q^2+q$, 
 hence $[X]^-=q$.
\end{proof}

\section{$\AGL{1}(\CC)$-representation varieties of torus knots}\label{sec:AGL1}

In this section we shall compute the motive of the $\AGL{1}(\CC)$-representation variety of the 
$(m,n)$-torus knot by describing it explicitly. Suppose that we have an element $(A, B) \in \Rep{m,n}(\AGL{1}(\CC))$ with matrices of the form
$$
	A = \begin{pmatrix}
	1 & 0 \\
	\alpha & a_0
	\end{pmatrix}, \qquad B = \begin{pmatrix}
	1 & 0 \\
	\beta & b_0
	\end{pmatrix}.
$$

A straightforward computation shows that
$$
	A^n = \begin{pmatrix}
	1 & 0 \\
	(1 + a_0 + \ldots+ a^{n-1}_0)\alpha & a^n_0
	\end{pmatrix}, \qquad B^m = \begin{pmatrix}
	1 & 0 \\
	(1 + b_0 + \ldots +b^{m-1}_0)\beta & b^m_0
	\end{pmatrix}.
$$
Notice that, since $\gcd(m,n) = 1$, for any pair $(a_0,b_0) \in \CC^2$ with $a^n_0 = b^m_0$ and $a_0,b_0 \neq 0$, there exists a unique
$t \in \CC^* = \CC-\left\{0\right\}$ such that $t^m = a_0$ and $t^n = b_0$. This means that the representation variety can be explicitly described as
$$
\Rep{m,n}(\AGL{1}(\CC)) = \left\{(t, \alpha, \beta) \in \CC^* \times \CC^2 \,\left|\, \Phi_n(t^m)\alpha = \Phi_m(t^n)\beta\right.\right\},
$$
where $\Phi_l$ is the polynomial
 $$
 \Phi_l(x) = 1 + x + \ldots + x^{l-1} = \frac{x^l-1}{x-1} \in \CC[x].
 $$ 

Written in a more geometric fashion, the morphism $(t, \alpha, \beta) \mapsto t$ defines a regular map
\begin{align}\label{eq:fibration-agl1}
	\Rep{m,n}(\AGL{1}(\CC)) \stackrel{}{\longrightarrow} \CC^*.
\end{align}
The fiber over $t \in \CC^*$ is the annihilator of the vector $(\Phi_n(t^m), \Phi_m(t^n)) \in \CC^2$ (in other words, the orthogonal
complement respect to the standard euclidean metric). This annihilator is $\CC$ if $(\Phi_n(t^m), \Phi_m(t^n)) \neq (0,0)$ and is $\CC^2$ otherwise.

Denote by $\mu_l$ the group of $l$-th roots of units. Recall that the roots of the polynomial $\Phi_l$ are the elements of 
$\mu_l^*=\mu_l-\{1\}$.
Hence $(\Phi_n(t^m), \Phi_m(t^n)) = (0,0)$ if and only if 
 $$
 t  \in \Omega_{m,n}=\mu_{mn} - \left(\mu_m \cup \mu_n\right). 
 $$
The number of elements of $\Omega_{m,n}$ is
 $$
 |\Omega_{m,n}| = mn-m-n+1=(m-1)(n-1).
 $$

The  space  (\ref{eq:fibration-agl1}) decomposes into the two Zariski locally trivial fibrations
\begin{align*}
	&\CC \longrightarrow \Rep{m,n}^{(1)}(\AGL{1}(\CC)) \longrightarrow \CC^* - \Omega_{m,n}, \\ 
	&\CC^2 \longrightarrow \Rep{m,n}^{(2)}(\AGL{1}(\CC)) \longrightarrow \Omega_{m,n},
\end{align*}
with $\Rep{m,n}(\AGL{1}(\CC))  = \Rep{m,n}^{(1)}(\AGL{1}(\CC)) \sqcup \Rep{m,n}^{(2)}(\AGL{1}(\CC))$. 
This implies that the motive of the whole representation variety is
\begin{align*}
	\left[\Rep{m,n}(\AGL{1}(\CC)) \right] &= \left[\Rep{m,n}^{(1)}(\AGL{1}(\CC)) \right] + \left[\Rep{m,n}^{(2)}(\AGL{1}(\CC)) \right] \\ 
	&= \left[\CC^* - \Omega_{m,n}\right]\left[\CC\right] + \left[\Omega_{m,n}\right]\left[\CC^2\right] \\
	& = (q-1 - |\Omega_{m,n}|)q + |\Omega_{m,n}|q^2 \\
	& = (mn-m-n+2)(q^2-q).
\end{align*}
This proves the first assertion of Theorem \ref{thm:main}.

\section{$\AGL{2}(\CC)$-representation varieties of torus knots}\label{sec:AGL2}

In this section, we compute the motive of the $\AGL{2}(\CC)$-representation variety of the $(m,n)$-torus knot. 
Suppose that we have an element $(A, B) \in \Rep{m,n}(\AGL{2}(\CC))$ with matrices of the form
$$
	A = \begin{pmatrix}
	1 & 0 \\
	\alpha & A_0
	\end{pmatrix}, \qquad B = \begin{pmatrix}
	1 & 0 \\
	\beta & B_0
	\end{pmatrix}.
$$
Notice that in this setting $A_0, B_0 \in \GL{2}(\CC)$ while $\alpha, \beta \in \CC^2$. Computing the powers we obtain
$$
	A^n = \begin{pmatrix}
	1 & 0 \\
	\Phi_n(A_0)\alpha & A^n_0
	\end{pmatrix}, \qquad B^m = \begin{pmatrix}
	1 & 0 \\
	\Phi_m(B_0)\beta & B^m_0
	\end{pmatrix}.
$$

Therefore, the $\AGL{2}(\CC)$-representation variety is explicitly given by
\begin{equation}\label{eq:AGL2-rep}
\Rep{m,n}(\AGL{2}(\CC)) = \left\{(A_0, B_0, \alpha, \beta) \in \GL{2}(\CC)^2 \times \CC^2 \,\left|\, \begin{matrix} A_0^n = B_0^m \\
\Phi_n(A_0)\alpha = \Phi_m(B_0)\beta \end{matrix}\right.\right\},
\end{equation}

In particular, these conditions imply that $(A_0, B_0) \in \Rep{m,n}(\GL{2}(\CC))$. Let us decompose 
$$
 \Rep{m,n}(\AGL{2}(\CC)) = \Rep{m,n}^{\irr}(\AGL{2}(\CC)) \sqcup \Rep{m,n}^{\red}(\AGL{2}(\CC)),
 $$
where $\Rep{m,n}^{\irr}(\AGL{2}(\CC))$ (resp.\ $\Rep{m,n}^{\red}(\AGL{2}(\CC))$) are the representations $(A, B)$ with $(A_0, B_0)$ an
irreducible (resp.\ reducible) representation of $\Rep{m,n}(\GL{2}(\CC))$.

\begin{remark}
Beware of the notation: the superscripts refer to the reducibility/irreducibility of the vectorial part of the representation, not to
the representation itself. %By its very definition, all the $\AGL{r}$-representations are reducible.
\end{remark}

\subsection{The irreducible stratum}\label{sec:irred-strat}

First of all, let us analyze the case where $(A_0, B_0)$ is an irreducible representation. In that case, the eigenvalues are restricted as the following result shows.

\begin{lemma}
Let $\rho = (A_0, B_0) \in \Rep{m,n}^{\irr}(\GL{r}(\CC))$ be an irreducible representation. Then $A_0^n = B_0^m = \omega\, \mathrm{Id}$, 
for some $\omega \in \CC^*$.
\end{lemma}

\begin{proof}
Notice that $A^n_0$ is a linear map that is equivariant with respect to the representation $\rho$. By Schur's lemma, this implies that
$A_0^n$ must be a multiple of the identity, say $A_0^n = \omega \, \mathrm{Id}$ and, since $B_0^m = A_0^n$, the result follows.
\end{proof}

\begin{corollary}\label{cor;1}
Let $\rho = (A_0, B_0) \in \Rep{m,n}^{\irr}(\GL{r}(\CC))$ be an irreducible representation and let $\lambda_1, \ldots, \lambda_r$ and
$\eta_1, \ldots, \eta_r$ be the eigenvalues of $A_0$ and $B_0$, respectively. Then $A_0$ and $B_0$ are diagonalizable and
$\lambda_1^n = \ldots = \lambda_r^n = \eta_1^m = \ldots = \eta_r^m$.
\end{corollary}

In order to analyze the conditions of (\ref{eq:AGL2-rep}), observe that $(A, B) \mapsto (A_0, B_0)$ defines a morphism
\begin{equation}\label{eq:map-fibration}
	\Rep{m,n}^{\irr}(\AGL{2}(\CC)) \longrightarrow \Rep{m,n}^{\irr}(\GL{2}(\CC)).
\end{equation}
The fiber of this morphism at $(A_0, B_0)$ is the kernel of the map
\begin{equation}\label{eq:map-lambda}
	\Lambda: \CC^2 \times \CC^2 \to \CC^2, \quad \Lambda(\alpha, \beta) = \Phi_n(A_0)\alpha - \Phi_m(B_0)\beta.
\end{equation}

The following appears in Proposition 7.3 in \cite{MP}{}. Recall from Example \ref{ex:calculations-kvar} that $[\PGL{2}(\CC)]=q^3-q$.

\begin{proposition} \label{prop:GL2}
For the torus knot of type $(m,n)$, we have:
\begin{itemize}
\item If $m,n$ are both odd then
$[\Rep{m,n}^{\irr}(\GL{2}(\CC))] = [\PGL{2}(\CC)] \frac14 (m-1)(n-1)(q-2) (q-1)$.
\item If $n$ is even and $m$ is odd, then

$[\Rep{m,n}^{\irr}(\GL{2}(\CC))] =  [\PGL{2}(\CC)] \left(\frac14 (n-2)(m-1)(q-2) +\frac12 (m-1)(q-1)\right) (q-1)$.
\item If $m$ is even and $n$ is odd, then

$[\Rep{m,n}^{\irr}(\GL{2}(\CC))] =  [\PGL{2}(\CC)] \left(\frac14 (n-1)(m-2)(q-2) +\frac12 (n-1)(q-1)\right) (q-1)$.
\end{itemize}
\end{proposition}

To understand the kernel of (\ref{eq:map-lambda}), we use the following lemma.

\begin{lemma}
Let $A$ be a diagonalizable matrix and let $p(x) \in \CC[x]$ a polynomial. Then, the dimension of the kernel of the matrix $p(A)$
is the number of eigenvalues of $A$ that are roots of $p(x)$.
\end{lemma}

\begin{proof}
Write $A = QDQ^{-1}$ with $D = \textrm{diag}(\lambda_1, \ldots, \lambda_r)$ a diagonal matrix. Then $p(A) = Qp(D)Q^{-1}$ and,
since $p(D) = \textrm{diag}(p(\lambda_1), \ldots, p(\lambda_r))$, 
the dimension of its kernel is the number of eigenvalues that are also roots of $p$.
\end{proof}

Using the previous lemma for $r=2$, we get that the dimension of the kernel of $\Phi_n(A_0)$ is the number of eigenvalues of $A_0$
that belong to $\mu_n^*$, and analogously for $\Phi_m(B_0)$. Let $\lambda_1, \lambda_2$ be the eigenvalues of $A_0$ and $\eta_1, \eta_2$
the eigenvalues of $B_0$. Recall that $\lambda_1 \neq \lambda_2$ and $\eta_1 \neq \eta_2$ since otherwise $(A_0, B_0)$ is not irreducible.
Then, we have the following options:

\begin{enumerate}

\item Case $\lambda_1, \lambda_2 \in \mu_n^*$ and $\eta_1, \eta_2 \in \mu_m^*$. In this situation, $\Lambda \equiv 0$ so
$\Ker{\Lambda} = \CC^4$. Hence, if we denote by $\Rep{m,n}^{\irr,(1)}(\AGL{2}(\CC))$ and $\Rep{m,n}^{\irr,(1)}(\GL{2}(\CC))$ 
the corresponding strata in (\ref{eq:map-fibration}) of the total and base space, respectively, we have that
	$$
 \left[\Rep{m,n}^{\irr,(1)}(\AGL{2}(\CC))\right] = \left[\Rep{m,n}^{\irr,(1)}(\GL{2}(\CC))\right][\CC^4].
	$$	
To get the motive of $\Rep{m,n}^{\irr,(1)}(\GL{2}(\CC))$, the eigenvalues define a fibration
\begin{equation}\label{eq:eigenvalues}
	\Rep{m,n}^{\irr,(1)}(\GL{2}(\CC)) \longrightarrow ((\mu_n^*)^2 - \Delta)/\ZZ_2 \times ((\mu_m^*)^2 - \Delta)/\ZZ_2,
\end{equation}
where $\Delta$ is the diagonal and $\ZZ_2$ acts by permutation of the entries. The fiber of this map is the collection of
representations $(A_0, B_0) \in \Rep{m,n}^{\irr}(\GL{2})$ with fixed eigenvalues, denoted by $\Rep{m,n}^{\irr}(\GL{2}(\CC))_{0}$. An element of $\Rep{m,n}^{\irr}(\GL{2}(\CC))_{0}$ 
is completely determined by the two pairs of eigenspaces of $(A_0, B_0)$ up to conjugation. Since the representation $(A_0, B_0)$
must be irreducible, these eigenspaces must be pairwise distinct. Hence, this variety is
$\Rep{m,n}^{\irr}(\GL{2}(\CC))_{0} = (\PP^1)^4 - \Delta_c$,
 where $\Delta_c \subset (\PP^1)^4$ is the `coarse diagonal' of tuples with two repeated entries. There is a free and closed action of $\PGL{2}(\CC)$ on $(\PP^1)^4$ with quotient 
	$$
 \frac{(\PP^1)^4 - \Delta_c}{\PGL{2}(\CC)} = \PP^1-\{0,1,\infty\}.
	$$ 
To see this, note that there is a $\PGL{2}(\CC)$-equivariant map that sends the first three entries to
$0,1,\infty \in \PP^1$ respectively, so the orbit is completely determined by the image of the fourth point under this map. 
Hence, $[\Rep{m,n}^{\irr}(\GL{2}(\CC))_{0}] = [\PP^1-\{0,1,\infty\}]\, [\PGL{2}(\CC)]=(q-2)(q^3-q)$. 
	
Coming back to the fibration (\ref{eq:eigenvalues}), we have that the basis is a set of 
$\binom{n-1}{2}\binom{m-1}{2}=\frac{(n-1)(n-2)(m-1)(m-2)}{4}$ points, so
	$$
 \left[\Rep{m,n}^{\irr,(1)}(\GL{2}(\CC))\right] = \frac{(n-1)(n-2)(m-1)(m-2)}{4}(q-2)(q^3-q),
	$$
and thus,
	$$
 \left[\Rep{m,n}^{\irr,(1)}(\AGL{2}(\CC))\right] = \frac{(n-1)(n-2)(m-1)(m-2)}{4}(q^5-2q^4)(q^3-q).
	$$
	
\item Case $\lambda_1, \lambda_2 \in \mu_n^*$, $\eta_1 \in \mu_m^*$ and $\eta_2 = 1$. In this case, $\Ker{\Lambda} = \CC^3$
and the base space is made of $\binom{n -1}{2} (m-1)$ copies of $\Rep{m,n}^{\irr}(\GL{2}(\CC))_{0}$. 
Hence, this stratum contributes
 \begin{align*}
 \left[\Rep{m,n}^{\irr,(2)}(\AGL{2}(\CC))\right] &= \frac{(n-1)(n-2)(m-1)}{2}\left[\PP^1-\left\{0,1,\infty\right\}\right]\, [\PGL{2}(\CC)]\,[\CC^3] \\
 &= \frac{(n-1)(n-2)(m-1)}{2}(q^4-2q^3)(q^3-q).
	\end{align*}

\item Case $\lambda_1 \in \mu_n^*$, $\lambda_2 = 1$ and $\eta_1,\eta_2 \in \mu_m^*$. This is analogous to the previous stratum and contributes
 \begin{align*}
 \left[\Rep{m,n}^{\irr,(3)}(\AGL{2}(\CC))\right] &= \frac{(m-1)(n-1)(m-2)}{2}\left[\PP^1-\left\{0,1,\infty\right\}\right]\, [\PGL{2}(\CC)]\, [\CC^3] \\
 &= \frac{(m-1)(n-1)(m-2)}{2}(q^4-2q^3)(q^3-q).
 \end{align*}

\item Case $\lambda_1 \in \mu_n^*$, $\lambda_2 = 1$ and $\eta_1 \in \mu_m^*$, $\eta_2 = 1$. Now, $\Ker{\Lambda} = \CC^2$ and this stratum contributes 
 \begin{align*}
  \left[\Rep{m,n}^{\irr,(4)}(\AGL{2}(\CC))\right] &= (m-1)(n-1)\left[\PP^1-\left\{0,1,\infty\right\}\right]\, [\PGL{2}(\CC)]\,[\CC^2] \\
  &= (m-1)(n-1)(q^3-2q^2)(q^3-q).
 \end{align*}

\item Case $\lambda_1 \not\in \mu_n^*, \lambda_2 \not\in \mu_n^*, \eta_1 \not\in \mu_m^*$ and $\eta_2 \not\in \mu_m^*$.
Recall that by Corollary \ref{cor;1}, these conditions are all equivalent.
In this situation, $\Lambda$ is surjective so $\Ker{\Lambda} = \CC^2$. The motive 
$\left[\Rep{m,n}^{\irr}(\GL{2}(\CC))\right]$ is given in Proposition \ref{prop:GL2}.
To this space, we have to remove the orbits corresponding to the forbidden eigenvalues, which are 
  \begin{align*}
 \ell_{m,n} =& \, \frac{(n-1)(n-2)(m-1)(m-2)}{4} + \frac{(n-1)(n-2)(m-1)}{2} \\ &+ \frac{(m-1)(n-1)(m-2)}{2} + (m-1)(n-1) =
\frac14 mn(m-1)(n-1) 
 \end{align*} 
 copies of $[\Rep{m,n}^{\irr}(\GL{2}(\CC))_{0}] = [\PP^1-\{0,1,\infty\}]\, [\PGL{2}(\CC)]$. Hence this stratum contributes
  \begin{align*}
 \left[\Rep{m,n}^{\irr,(5)}(\AGL{2}(\CC))\right]  =& \left(\left[\Rep{m,n}^{\irr}(\GL{2}(\CC))\right] - \ell_{m,n}(q-2)(q^3-q)\right)\left[\CC^2\right]
  \\
  =& \left[\Rep{m,n}^{\irr}(\GL{2}(\CC))\right] q^2-  \frac14 mn(m-1)(n-1) (q^3-2q^2)(q^3-q).
 \end{align*} 

\end{enumerate}

Adding up all the contributions, we get 
\begin{align*}
	\left[\Rep{m,n}^{\irr}(\AGL{2}(\CC))\right] =& \left[\Rep{m,n}^{\irr,(1)}(\AGL{2}(\CC))\right] + \left[\Rep{m,n}^{\irr,(2)}(\AGL{2}(\CC))\right]
	+ \left[\Rep{m,n}^{\irr,(3)}(\AGL{2}(\CC))\right] \\
	& + \left[\Rep{m,n}^{\irr,(4)}(\AGL{2}(\CC))\right] + \left[\Rep{m,n}^{\irr,(5)}(\AGL{2}(\CC))\right]  \\
	= & \frac{(m-1)(n-1)(q^3-2q^2)(q-1)(q^3-q)}{4}  \left( (m-2)(n-2)q+mn-4 \right) \\ &
+\left[\Rep{m,n}^{\irr}(\GL{2}(\CC))\right] q^2 .
\end{align*}

%%%%%%%%%%%%%%%%%%%%%%%%%%%%%%%%%%%%
\subsection{The reducible stratum}
%%%%%%%%%%%%%%%%%%%%%%%%%%%%%%%%%%%%

In this section, we shall consider the case in which $(A_0, B_0) \in \Rep{m,n}^{\red}(\GL{2}(\CC))$ 
is a reducible representation. 
After a change the basis, since $A_0^n = B_0^m$, we can suppose that $(A_0, B_0)$ has exactly one of the following three forms:
 $$
 \textrm{(A)} \left(\begin{pmatrix}t_1^m & 0 \\ 0 & t_2^m\end{pmatrix}, \begin{pmatrix}t_1^n & 0 \\ 0 & t_2^n\end{pmatrix}\right), 
   \textrm{(B)} \left(\begin{pmatrix}t^m & 0 \\ 0 & t^m\end{pmatrix}, \begin{pmatrix}t^n & 0 \\ 0 & t^n\end{pmatrix}\right),
 \textrm{(C)} \left(\begin{pmatrix}t^m & 0 \\ x & t^m\end{pmatrix}, \begin{pmatrix}t^n & 0 \\ y & t^n\end{pmatrix}\right),
 $$
with $t_1, t_2, t \in \CC^*$, $x, y \in \CC$ and satisfying $t_1 \neq t_2$ and $(x,y) \neq (0,0)$.

Restricting to the representations of each stratum $S = (\textrm{A}), (\textrm{B}), (\textrm{C})$, 
we have a morphism
 \begin{equation}\label{eq:map-fibration-A}
 \Rep{m,n}^{S}(\AGL{2}(\CC)) \longrightarrow \Rep{m,n}^{S}(\GL{2}(\CC)),
 \end{equation}
whose fiber is the kernel of the linear map (\ref{eq:map-lambda}).

\subsubsection{Case \textrm{(A)}}\label{sec:red-A}
In this case, as for the irreducible part of Section \ref{sec:irred-strat}, the kernel of $\Lambda$ depends on whether $t_1, t_2$ are
roots of the polynomial $\Phi_l$. In this case the base space is
 $$
 \Rep{m,n}^{ \textrm{(A)}}
 (\GL{2}(\CC)) = \left(\left((\CC^*)^2 - \Delta\right) \times \frac{\GL{2}(\CC)}{\GL{1}(\CC) \times \GL{1}(\CC)}\right)/\ZZ_2,
 $$
with the action of $\ZZ_2$ given by exchange of eigenvalues and eigenvectors. Using Lemma \ref{lem;ac} and (\ref{eqn:2bis}), 
we have
 \begin{align*}
 \left[\Rep{m,n}^{\mathrm{(A)}}(\GL{2})\right] &=
 [(\CC^*)^2 - \Delta]^+ 
  \left[\frac{\GL{2}(\CC)}{\GL{1}(\CC) \times \GL{1}(\CC)}\right]^+ 
  + [(\CC^*)^2 - \Delta]^-
   \left[\frac{\GL{2}(\CC)}{\GL{1}(\CC) \times \GL{1}(\CC)}\right]^- \\
 &= q^2(q-1)^2 - q(q-1).
\end{align*}

On the other hand, if we fix the eigenvalues of $(A_0, B_0)$ as in Section \ref{sec:irred-strat}, 
the corresponding fiber $\Rep{m,n}^{\mathrm{(A)}}(\GL{2}(\CC))_0$ is
$$
 \left[\Rep{m,n}^{\mathrm{(A)}}(\GL{2}(\CC))_0\right] = \left[\frac{\GL{2}(\CC)}{\GL{1}(\CC) \times \GL{1}(\CC)}\right] = q^2 + q.
$$

As in Section \ref{sec:AGL1}, set $\Omega_{m,n} = \mu_{mn} - (\mu_{m} \cup \mu_{n})$ for those 
$t \in \CC^*$ such that $\Phi_n(t^m) = 0$ and $\Phi_m(t^n) = 0$. With this information at hand, we compute for each stratum:

\begin{enumerate}
 \item Case $t_1, t_2 \in \Omega_{m,n}$. In this situation, $\Lambda \equiv 0$ so $\Ker{\Lambda} = \CC^4$. 
  The eigenvalues yield a fibration
\begin{equation*}
	\Rep{m,n}^{\mathrm{(A)},(1)}(\GL{2}(\CC)) \longrightarrow \left(\Omega_{m,n}^2 - \Delta\right)/\ZZ_2
\end{equation*}
whose fiber is $\Rep{m,n}^{\mathrm{(A)}}(\GL{2}(\CC))_0$. Observe that $\left(\Omega_{m,n}^2 - \Delta\right)/\ZZ_2$ 
is a finite set of $(m-1)(n-1)((m-1)(n-1)-1)/2$ points, so we have
\begin{align*}
	\left[\Rep{m,n}^{\mathrm{(A)},(1)}(\AGL{2}(\CC))\right] &= \left[\Rep{m,n}^{\mathrm{(A)},(1)}(\GL{2}(\CC))\right][\CC^4] \\
	&= \left[\Rep{m,n}^{\mathrm{(A)}}(\GL{2}(\CC))_0\right][\CC^4]\left[\left(\Omega_{m,n}^2 - \Delta\right)/\ZZ_2\right] \\
	&= \frac{(m-1)(n-1)(mn-m-n)}{2}q^4(q^2+q).
\end{align*}

\item Case $t_1 \in \Omega_{m,n}$ but $t_2 \not\in \Omega_{m,n}$ (or vice-versa, the order is not important here).
Now, we have a locally trivial fibration
$$
 \Rep{m,n}^{\mathrm{(A)},(2)}(\GL{2}(\CC)) \longrightarrow \Omega_{m,n} \times \left(\CC^*-\Omega_{m,n}\right),
$$
with fiber $\Rep{m,n}^{\mathrm{(A)}}(\GL{2}(\CC))_0$. The kernel of $\Lambda$ is $\CC^3$, so this stratum contributes
\begin{align*}
 \left[\Rep{m,n}^{\mathrm{(A)},(2)}(\AGL{2}(\CC))\right] = (m-1)(n-1)(q - mn+n+m-2) q^3(q^2+q).
\end{align*}

\item Case $t_1, t_2 \not\in\Omega_{m,n}$. The kernel is now $\CC^2$ and we have a fibration
$$
 \Rep{m,n}^{\mathrm{(A)},(3)}(\GL{2}(\CC)) \longrightarrow B,
$$
where the motive of the base space $B$ is 
\begin{align*}
 [B] &=[(\CC^*)^2-\Delta]^+  - \left[\Omega_{m,n}^2 - \Delta\right]^+ 
- \left[\Omega_{m,n}\right] \left(q-1-[\Omega_{m,n}]\right) =\\
 &=  (q-1)^2 -\frac{(m-1)(n-1)(mn-m-n)}{2}- (m-1)(n-1)(q-mn+n+m-2) \\
 &= q^2 - (mn-m-n+3)q -\frac14 (m-1)(n-1)(mn-8).
\end{align*}
Therefore, this space contributes
\begin{align*}
 \big[\Rep{m,n}^{\mathrm{(A)},(3)}&(\AGL{2}(\CC))\big] = 
 \left[\Rep{m,n}^{\mathrm{(A)},(3)}(\GL{2}(\CC))\right][\CC^2]  \\
 &= q^2(q^2+q)\left(q^2 - (mn-m-n+3)q+ (m-1)(n-1)(mn-8)/4\right).
\end{align*}
\end{enumerate} 

Adding up all the contributions, we get that

\begin{align*}
 \left[\Rep{m,n}^{\mathrm{(A)}}(\AGL{2}(\CC))\right] =\; &
 (q^2+q)q^2\bigg(\frac{(m-1)(n-1)(mn-m-n)}{2}(q^2-1) \\
 & + (m-1)(n-1)(q - mn+n+m-2) (q-1) + (q-1)^2\bigg).
 \end{align*}

\subsubsection{Case \textrm{(B)}}\label{sec:red-B} In this setting, this situation is simpler. Observe that the adjoint
action of $\GL{2}(\CC)$ on the vectorial part is trivial, so the corresponding $\GL{2}(\CC)$-representation variety is just
$$
	\Rep{m,n}^{\mathrm{(B)}}(\GL{2}(\CC)) = \CC^*.
$$
Analogously, the variety with fixed eigenvalues, $\Rep{m,n}^{\mathrm{(B)}}(\GL{2}(\CC))_0$ is just a point.
With these, we obtain that:
\begin{enumerate}
\item If $t \in \Omega_{m,n}$, then $\Ker{\Lambda} = \CC^4$. We have a fibration
\begin{equation*}
 \Rep{m,n}^{\mathrm{(B)},(1)}(\GL{2}(\CC)) \longrightarrow \Omega_{m,n}
\end{equation*}
whose fiber is $\Rep{m,n}^{\mathrm{(B)}}(\GL{2}(\CC))_0$. Hence, this stratum contributes
$$
 \left[\Rep{m,n}^{\mathrm{(B)},(1)}(\AGL{2}(\CC))\right] =
  \left[\Rep{m,n}^{\mathrm{(B)},(1)}(\GL{2}(\CC))\right][\CC^4] = (m-1)(n-1)q^4\, .
$$

\item If $t \not\in \Omega_{m,n}$, then $\Ker{\Lambda} = \CC^2$. We have a fibration
\begin{equation*}
 \Rep{m,n}^{\mathrm{(B)},(2)}(\GL{2}(\CC)) \longrightarrow \CC^*-\Omega_{m,n}.
\end{equation*}
Thus, the contribution of this stratum is
$$
 \left[\Rep{m,n}^{\mathrm{(B)},(2)}(\AGL{2}(\CC))\right] = (q-1-(m-1)(n-1))q^2.
$$
\end{enumerate}
The total contribution is
 $$
  \left[\Rep{m,n}^{\mathrm{(B)}}(\AGL{2}(\CC))\right] = (m-1)(n-1)(q^4-q^2)+(q-1)q^2.
$$

\subsubsection{Case (C)}\label{sec:red-C} 
In this case, an extra calculation must be done to control the off-diagonal entry. If $(A_0, B_0)$ has the form
$$ \left(\begin{pmatrix}t^m & 0 \\ x & t^m\end{pmatrix}, \begin{pmatrix}t^n & 0 \\ y & t^n\end{pmatrix}\right),
$$
then the condition $A_0^n = B_0^m$ reads as
$$
 \begin{pmatrix}t^{mn} & 0 \\ nt^{m(n-1)}x & t^{mn}\end{pmatrix} = \begin{pmatrix}t^{mn} & 0 \\ mt^{n(m-1)}y & t^{mn}\end{pmatrix}.
$$
The later conditions reduce to $nt^{m(n-1)}x = mt^{n(m-1)}y$ and, since $t \neq 0$, this means that $(x,y)$ should lie in a
line minus $(0,0)$. The stabilizer of a Jordan type matrix in $\GL{2}(\CC)$ is the subgroup $U = (\CC^*)^2 \times \CC \subset \GL{2}(\CC)$
of upper triangular matrices. Hence, the corresponding $\GL{2}(\CC)$-representation variety is
$$
	\Rep{m,n}^{\mathrm{(C)}}(\GL{2}(\CC)) = (\CC^*)^2 \times \GL{2}(\CC)/U.
$$
In particular, $\left[\Rep{m,n}^{\mathrm{(C)}}(\GL{2}(\CC))\right] 
= (q-1)^2 (q^4-q^3-q^2+q)/q(q-1)^2 = (q-1)^2(q+1)$. Moreover, if we fix 
the eigenvalues we get that $\left[\Rep{m,n}^{\mathrm{(C)}}(\GL{2}(\CC))_0\right] = \left[\CC^* \times \GL{2}(\CC)/U\right] = (q-1)(q+1)$.

To analyze the condition $\Phi_n(A_0) = \Phi_m(B_0)$, a straightforward computation reduces it to
$$
	\begin{pmatrix}\Phi_n(t^m) & 0 \\ \displaystyle{x\sum_{i=1}^{n-1} it^{m(i-1)}} & \Phi_n(t^m)\end{pmatrix} = \begin{pmatrix}\Phi_m(t^n) & 0 
	\\ \displaystyle{y\sum_{i=1}^{m-1} it^{n(i-1)}} & \Phi_m(t^n)\end{pmatrix}.
$$
The off-diagonal entries can be recognized as $x\Phi_n'(t^m)$ and $y\Phi_m'(t^n)$ respectively, where $\Phi_l'(x)$ denotes the formal
derivative of $\Phi_l(x)$. Since $\Phi_l$ has no repeated roots, we have that $\Phi_n'(t^m)$ and $\Phi_n(t^m)$ (resp.\ $\Phi_m'(t^n)$
and $\Phi_m(t^n)$) cannot vanish simultaneously. Therefore, stratifying according to the kernel of $\Lambda$ we get the following two possibilities:

\begin{enumerate}
\item If $t \in \Omega_{m,n}$, then $\Ker{\Lambda} = \CC^3$. We have a fibration
\begin{equation*}
 \Rep{m,n}^{\mathrm{(C)},(1)}(\GL{2}(\CC)) \longrightarrow \Omega_{m,n}
\end{equation*}
whose fiber is $\Rep{m,n}^{\mathrm{(C)}}(\GL{2}(\CC))_0$. Hence, this stratum contributes
$$
 \left[\Rep{m,n}^{\mathrm{(C)},(1)}(\AGL{2}(\CC))\right] = \left[\Rep{m,n}^{\mathrm{(C)},(1)}(\GL{2}(\CC))\right][\CC^3] = (m-1)(n-1)q^3(q-1)(q+1).
$$

\item If $t \in \CC^*-\Omega_{m,n}$, then $\Ker{\Lambda} = \CC^2$. The fibration we get is now
\begin{equation*}
	\Rep{m,n}^{\mathrm{(C)},(2)}(\GL{2}(\CC)) \longrightarrow \CC^*-\Omega_{m,n}\, .
\end{equation*}
Therefore, this stratum contributes
\begin{align*}
 \left[\Rep{m,n}^{\mathrm{(C)},(2)}(\AGL{2}(\CC))\right] &= \left[\Rep{m,n}^{\mathrm{(C)},(2)}(\GL{2}(\CC))\right][\CC^2] \\ 
 &= \left[\Rep{m,n}^{\mathrm{(C)}}(\GL{2}(\CC)) - (m-1)(n-1)\Rep{m,n}^{\mathrm{(C)},(2)}(\GL{2}(\CC))_0\right][\CC^2] \\
 &= \left((q-1)^2(q+1) - (m-1)(n-1)(q-1)(q+1)\right)q^2.
\end{align*}
\end{enumerate}

Adding up all the contributions, we get that
$$
 \left[\Rep{m,n}^{\mathrm{(C)}}(\AGL{2}(\CC))\right] = (q-1)^2(q+1)q^2 + (m-1)(n-1)(q-1)(q+1)(q^3-q^2).
$$

Putting the results of Sections \ref{sec:irred-strat}, \ref{sec:red-A}, \ref{sec:red-B} and \ref{sec:red-C} together, we prove the
second formula in Theorem \ref{thm:main}.

\section{Character varieties of torus knots}

As we have said in Section \ref{sec:rep-var}, the $G$-representation variety of a $(m,n)$-torus knot $\Rep{m,n}(G)$ parametrizes
all the representations $\rho: \pi_1(\RR^3 - K_{m,n}) \to G$. However, this space does not take into account the fact that two
representations might be isomorphic. To remove this redundancy, consider the adjoint action of $G$ on $\Rep{m,n}(G)$ given by 
$(P \cdot \rho)(\gamma) =P\rho(\gamma)P^{-1}$ for $P \in G$, $\rho \in \Rep{m,n}(G)$ and $\gamma \in \pi_1(\RR^3 - K_{m,n})$.

Ideally, we would like to take the quotient space $\Rep{m,n}(G)/G$ as the moduli space of isomorphism classes of representations.
However, typically this orbit space is not an algebraic variety, and we need to consider instead the Geometric Invariant Theory (GIT) quotient \cite{Ne}
$$
	\Char{m,n}(G) = \Rep{m,n}(G) \sslash G,
$$
usually known as the \emph{character variety}. Roughly speaking, the character variety is obtained by collapsing those orbits of
isomorphism classes of representations of the representation variety whose Zariski closures intersect. This collapsing can be
justified intuitively since those orbits are indistinguishable from the point of their structure sheaf. 

In the case that $G$ is affine (so that $\Rep{m,n}(G)$ is also an affine variety), there is a very simple description of the
GIT quotient. Let $\cO(\Rep{m,n}(G))$ be the ring of regular functions on $\Rep{m,n}(G)$ (the global sections of its structure sheaf).
The action of $G$ on $\Rep{m,n}(G)$ induces an action on $\cO(\Rep{m,n}(G))$. Set $\cO(\Rep{m,n}(G))^{G}$ for the collection
of $G$-invariant functions. By Nagata's theorem \cite{Na}{}, if $G$ is a reductive group then this is a finitely generated algebra
so we can take as the GIT quotient the algebraic variety
$$
	\Char{m,n}(G) = \Rep{m,n}(G) \sslash G = \textrm{Spec}\left(\cO(\Rep{m,n}(G))^{G}\right).
$$

This is the construction of character varieties that is customarily developed in the literature for the classical groups
$G = \GL{r}(\CC), \SL{r}(\CC)$. However, the affine case $G = \AGL{r}(\CC)$ is problematic since $\AGL{r}(\CC)$ is not a
reductive group. Roughly speaking, the underlying reason is that we have a description as semi-direct product
$\AGL{r}(\CC) = \CC^{r} \rtimes \GL{r}(\CC)$ and the factor $\CC^{r}$ is the canonical example of a non-reductive group.

For this reason, it is not guaranteed by Nagata's theorem that $\cO(\Rep{m,n}(\AGL{r}(\CC)))^{\AGL{r}(\CC)}$ is a finitely
generated algebra so the GIT quotient may not be defined as an algebraic variety. However, in this situation we have the following result.

\begin{proposition}
For any $r \geq 1$ we have that
$$
	\cO(\Rep{m,n}(\AGL{r}(\CC)))^{\AGL{r}(\CC)} = \cO(\Rep{m,n}(\GL{r}(\CC)))^{\GL{r}(\CC)}.
$$
\begin{proof}
We shall explode the natural description of $\Rep{m,n}(\GL{r}(\CC))$ as a subvariety of the whole representation variety
$\Rep{m,n}(\AGL{r}(\CC))$. By restriction, there is a natural homomorphism
$$
	\varphi: \cO(\Rep{m,n}(\AGL{r}(\CC)))^{\AGL{r}(\CC)} \longrightarrow \cO(\Rep{m,n}(\GL{r}(\CC)))^{\GL{r}(\CC)}.
$$
Notice that the action of $\AGL{r}(\CC)$ on the subvariety $\Rep{m,n}(\GL{r}(\CC))$ agrees with the $\GL{r}(\CC)$-action.
Hence, given an invariant function $f \in \cO(\Rep{m,n}(\GL{r}(\CC)))^{\GL{r}(\CC)}$ we can consider the lifting 
$\tilde{f} \in \cO(\Rep{m,n}(\AGL{r}(\CC)))^{\AGL{r}(\CC)}$ given by $\tilde{f}(A,B) = f(A_0, B_0)$ where $(A_0, B_0)$
is the vectorial part of the representation $(A,B) \in \Rep{m,n}(\AGL{r}(\CC))$. The map $f \mapsto \tilde{f}$ gives a right inverse to $\varphi$.

To show that this morphism is also a left inverse, let $(A,B) \in \Rep{m,n}(\AGL{r}(\CC))$, say
$$
	(A,B) = \left(\begin{pmatrix} 1 & 0 \\ \alpha & A_0\end{pmatrix}, \begin{pmatrix} 1 & 0 \\
	\beta & B_0\end{pmatrix}\right),
$$
with $A_0, B_0 \in \GL{r}(\CC)$ and $\alpha, \beta \in \CC^r$. Consider the homothety
$$
	P = \begin{pmatrix} 1 & 0 \\ 0 & \lambda\, \Id\end{pmatrix} \in \AGL{r}(\CC).
$$
Then, we have that
$$
	P \cdot (A,B) = \left(\begin{pmatrix} 1 & 0 \\ \lambda\alpha & A_0\end{pmatrix}, \begin{pmatrix} 1 & 0 \\
	\lambda\beta & B_0\end{pmatrix}\right).
$$
By letting $\lambda \to 0$, this implies that the Zariski closure of the orbit contains the representation
$$
	\left(\begin{pmatrix} 1 & 0 \\ 0 & A_0\end{pmatrix}, \begin{pmatrix} 1 & 0 \\
	0 & B_0\end{pmatrix}\right) \in \Rep{m,n}(\GL{r}(\CC)).
$$

Now, observe that any $\AGL{r}(\CC)$-invariant function $f: \Rep{m,n}(\AGL{r}(\CC)) \to \CC$ must take the same
value on the closure of an orbit, so for any $(A, B) \in \Rep{m,n}(\AGL{r}(\CC))$ we have that $f(A,B) = f(A_0, B_0)$.
In particular, this shows that $f \mapsto \tilde{f}$ is also a left inverse of $\varphi$, so $\varphi$ is an isomorphism.
\end{proof}
\end{proposition}

\begin{remark}
In fact, there is nothing special in considering torus knots in the previous proof. Exactly the same argument actually
proves that we have $\cO(\Rep{\G}(\AGL{r}(\CC)))^{\AGL{r}(\CC)} = \cO(\Rep{\G}(\AGL{r}(\CC)))^{\AGL{r}(\CC)}$ for the
representation variety of representations $\rho: \G\to \AGL{r}(\CC)$  for any finitely presented group $\G$.
\end{remark}

In particular, the previous proof shows that $\cO(\Rep{m,n}(\AGL{r}(\CC)))^{\AGL{r}(\CC)}$ is a finitely generated algebra,
so we can harmlessly define the $\AGL{r}(\CC)$-character variety and it satisfies 
 $$
 \Char{m,n}(\AGL{r}(\CC)) = \Char{m,n}(\GL{r}(\CC)).
 $$ 
The motive of the $\GL{r}(\CC)$-character variety has been previously computed in the literature for low rank $r$, 
for instance in \cite{Mu} for $r = 2$  (cf.\ Proposition \ref{prop:GL2}) and in \cite{MP} for $r = 3$.

%%%%%%%%%%%%%%%%%%%%%%%%%%%%%%%%%%%%%%%%%%%%%%%%%%%%%%%%%%%%%%%%
%%%%%%%%%%%%%%%%%%%%%%% BIBLIOGRAPHY %%%%%%%%%%%%%%%%%%%%%%%%%%%
%%%%%%%%%%%%%%%%%%%%%%%%%%%%%%%%%%%%%%%%%%%%%%%%%%%%%%%%%%%%%%%%

%\nocite{*}
%\bibliography{bibliography.bib}{}
%\bibliographystyle{plain}

\end{document}